\documentclass[10pt]{amsart}
\usepackage{amsthm}
\usepackage{amssymb}
\usepackage{epic,eepic}
\usepackage[mathscr]{eucal} 
\usepackage{comment}

\numberwithin{equation}{section}
\theoremstyle{theorem}
\newtheorem{theorem}{Theorem}[section]
\newtheorem{lemma}[theorem]{Lemma}
\newtheorem{proposition}[theorem]{Proposition}
\newtheorem{corollary}[theorem]{Corollary}
\newtheorem{conjecture}[theorem]{Conjecture}
\theoremstyle{definition}
\newtheorem{definition}[theorem]{Definition}
\newtheorem{example}[theorem]{Example}
\newtheorem{remark}[theorem]{Remark}

\title[Dilogarithm identities]
{Dilogarithm identities
 for\\ conformal field theories
and cluster algebras:\\ simply laced case}

\author[T.\ Nakanishi]{Tomoki Nakanishi}
\address{ Tomoki Nakanishi:
 Graduate School of Mathematics, Nagoya University,
Nagoya, 464-8604, Japan}

\begin{document}

\begin{abstract}
The dilogarithm identities
for the central charges of conformal
field theories of  simply laced type were conjectured
by Bazhanov, Kirillov, and Reshetikhin.
Their functional generalizations were conjectured by Gliozzi
and Tateo.
They have been partly proved by various authors.
We prove these identities in full generality
for any pair of  Dynkin diagrams of simply laced type
based on the cluster algebra formulation
of the Y-systems.
\end{abstract}

\maketitle

\section{Introduction}

\subsection{Dilogarithm identities}

Let $L(x)$ be the {\em Rogers dilogarithm function\/}
\cite{L,Ki2,Zag2,N}
\begin{align}
\label{eq:L0}
L(x)=-\frac{1}{2}\int_{0}^x 
\left\{ \frac{\log(1-y)}{y}+
\frac{\log y}{1-y}
\right\} dy
\quad (0\leq x\leq 1).
\end{align}
It is well known that the following properties hold
$(0\leq x,y\leq 1)$.
\begin{gather}
\label{eq:L1}
L(0)=0,
\quad L(1)=\frac{\pi^2}{6},\\
\label{eq:L2}
\quad L(x) + L(1-x)=\frac{\pi^2}{6},\\
\label{eq:L3}
\quad L(x) + L(y)+ L(1-xy)+
L\left( \frac{1-x}{1-xy}\right)
+L\left( \frac{1-y}{1-xy}\right)
=\frac{\pi^2}{2}.
\end{gather}

In the series of works by
Bazhanov, Kirillov, and Reshetikhin
\cite{KR1,BR1,KR2,Ki1,BR2} they reached a remarkable
conjecture on identities expressing the central charges
of conformal field theories
in terms of $L(x)$,
and partly established it.
Let us concentrate on the identities
in the {\em simply laced case} here.

Let $X_r$ be any simply laced
Dynkin diagram of finite type with rank $r$,
and $I$ be the index set of $X_r$.
Let $\ell\geq 2$ be any integer.
For a family of positive real numbers
$\{Y^{(a)}_m \mid a\in I; 1\leq m \leq \ell-1\}$,
consider a system of algebraic relations
\begin{align}
\label{eq:y1}
(Y^{(a)}_m)^2 = 
\frac{\displaystyle
\prod_{b:b\sim a} (1+Y^{(b)}_m)}
{(1+Y^{(a)}_{m-1}{}^{-1})(1+Y^{(a)}_{m+1}{}^{-1})},
\end{align}
where $b\sim a$ means $b$ is adjacent to $a$ in $X_r$,
and $Y^{(a)}_{0}{}^{-1}=Y^{(a)}_{\ell}{}^{-1}=0$
if they appear in the right hand sides.

\begin{theorem}[{\cite{NK,Zag2}}]
There exists a
unique positive real solution
of  \eqref{eq:y1}.
\end{theorem}

\begin{conjecture}[Dilogarithm identities \cite{Ki1,BR2}]
\label{conj:DI}
Suppose that a family of positive real numbers
$\{Y^{(a)}_m \mid a\in I; 1\leq m \leq \ell-1\}$
satisfies \eqref{eq:y1}.
Then, we have the identities
\begin{align}
\label{eq:DI}
\frac{6}{\pi^2}\sum_{a\in I}
\sum_{m=1}^{\ell-1}
L\left(\frac{Y^{(a)}_m}{1+Y^{(a)}_m}\right)
=
\frac{\ell \dim \mathfrak{g}}{h+\ell} - r,
\end{align}
where $h$ and
$\mathfrak{g}$ are  the Coxeter number
and  the simple Lie algebra of type $X_r$, respectively.
\end{conjecture}

\begin{remark}
The conjectures in \cite{Ki1} and \cite{BR2} are not exactly the same.
This is the version in \cite[Eqs.~(4.21), (4.22), (5.3)]{BR2}
with the identification of $f^a_j$ therein with
$Y^{(a)}_j/(1+Y^{(a)}_j)$ here.
The version in \cite{Ki1} also concerns the 
construction of the solution of \eqref{eq:y1}.
We do not touch this issue here, since it is regarded as
an independent problem in the framework of the present paper.
\end{remark}

For $X_r=A_r$, Kirillov \cite{Ki1} gave
the explicit expression of the solution
of \eqref{eq:y1}, and  proved the corresponding
identity \eqref{eq:DI} by the analytic method.

Due to the well-known
formula $\dim\mathfrak{g}=r(h+1)$, the right hand side of 
\eqref{eq:DI} is
 equal to the number
\begin{align}
\label{eq:c1}
\frac{ (\ell-1)rh}{h+\ell}.
\end{align}
It is already remarkable that the left hand side
of \eqref{eq:DI} is a rational number.
It is much more remarkable that
the rational number of the first term in the right hand side
of \eqref{eq:DI} is the central charge of the {\em Wess-Zumino-Witten
 conformal
field theory\/} \cite{KZ,GW} of type $X_r$ with level $\ell$.
(This is one of the reason
why the integer $\ell$ is called the {\em level}.)
The rational number in the right hand side of \eqref{eq:DI}
 itself
is also the central charge of the {\em parafermion conformal
field theory\/} of type $X_r$ with level $\ell$ \cite{FZ,G}.
The identity \eqref{eq:DI} is crucial  to establish the connection between
conformal field theories and various types of
nonconformal integrable models in various limits.

\begin{example}[\cite{KR1}]
 Consider the case $X_r=A_1$ and any $\ell$,
which is equivalent to the case
$X_r=A_{\ell-1}$ and $\ell=2$ by the {\em level-rank duality}.
Then, one has a solution
of \eqref{eq:y1}
\begin{align}
Y^{(1)}_m = \frac{\sin^2\frac{\pi}{\ell+2}}
    {\sin\frac{m\pi}{\ell+2}
\sin\frac{(m+2)\pi}{\ell+2}},
\end{align}
and the corresponding identity  \eqref{eq:DI} reads 
\begin{align}
\label{eq:DIex}
\frac{6}{\pi^2}
\sum_{m=1}^{\ell-1}
L\left(
\frac{
\sin^2 \frac{\pi}{\ell+2}
}
{
\sin^2\frac{(m+1)\pi}{\ell+2}
}
\right)
=
\frac{3\ell}{2+\ell} - 1.
\end{align}
This identity has been known and studied
by various authors in various points of view.
To name few, Lewin \cite{L}, Richmond-Szekeres \cite{RS},
Kirillov-Reshetikhin \cite{KR1},
Nahm-Recknagel-Terhoeven \cite{NRT},
Dupont-Sah \cite{DS}, {\em etc}.
\end{example}

\subsection{Functional dilogarithm identities}

The system \eqref{eq:y1} admits
an affinization called the {\em Y-system} introduced
by Zamolodchikov \cite{Zam}, Kuniba-Nakanishi \cite{KN},
and Ravanini-Tateo-Valleriani \cite{RTV}.
Here, we consider the version in \cite{RTV}.
Let $X_r$ and $X'_{r'}$ be a pair of simply laced
Dynkin diagrams of finite type.
Let $I$ and $I'$ be the
index sets of $X_r$ and $X'_{r'}$, respectively.
For a family of variables
 $\{Y_{ii'} (u)\mid i\in I, i'\in I',
u\in \mathbb{Z}\}$,
the {\em Y-system $\mathbb{Y}(X_r,X'_{r'})$ associated with
a pair $(X_r,X'_{r'})$}
is a system of the algebraic/functional relations
\begin{align}
\label{eq:Y2}
Y_{ii'}(u-1)Y_{ii'}(u+1)= 
\frac{\displaystyle
\prod_{j:j\sim i} (1+Y_{ji'}(u))}
{\displaystyle
\prod_{j':j'\sim i'} (1+Y_{ij'}(u)^{-1})},
\end{align}
where $j\sim i$ means $j$ is adjacent to $i$ in $X_r$,
while $j'\sim i'$ means $j'$ is adjacent to $i'$ in 
$X'_{r'}$.
Two systems $\mathbb{Y}(X_r,X'_{r'})$
and $\mathbb{Y}(X'_{r'},X_{r})$ are equivalent to
each other by the correspondence
$Y_{ii'}(u)\leftrightarrow Y_{i'i}(u)^{-1}$.
This is a generalization of the level-rank duality. 

Ravanini-Tateo-Valleriani \cite{RTV} gave
the periodicity conjecture, which  generalized the one
by \cite{Zam} in the case $X_r=A_1$ or $X'_{r'}=A_1$.

\begin{conjecture}[Periodicity \cite{RTV}]
\label{conj:period1}
Suppose that 
a family of positive real numbers
 $\{Y_{ii'} (u)\mid i\in I, i'\in I',
u\in \mathbb{Z}\}$
satisfies $\mathbb{Y}(X_r,X'_{r'})$.
Then, we have the periodicity
\begin{align}
Y_{ii'}(u+2(h+h'))=Y_{ii'}(u),
\end{align}
where $h$ and $h'$ are the Coxeter numbers of types
$X_r$ and $X'_{r'}$, respectively.
\end{conjecture}
Conjecture \ref{conj:period1} was proved
for $(X_r,X'_{r'})=(A_r,A_1)$
by
Gliozzi-Tateo \cite{GT2} and
Frenkel-Szenes \cite{FS},
for $(X_r,X'_{r'})=(\mbox{any},A_1)$
by Fomin-Zelevinsky \cite{FZ3},
and for $(X_r,X'_{r'})=(A_r,A_{r'})$
by Volkov \cite{V} and Szenes \cite{S}.
More recently, it was proved 
in full generality by Keller \cite{Kel1,Kel2}.

Furthermore, Gliozzi-Tateo \cite{GT1} significantly
generalized Conjecture \ref{conj:DI} as follows:

\begin{conjecture}[Functional dilogarithm identities \cite{GT1}]
\label{conj:DI2}
Suppose that 
a family of positive real numbers
 $\{Y_{aa'} (u)\mid a\in I, a'\in I',
u\in \mathbb{Z}\}$
satisfies $\mathbb{Y}(X_r,X'_{r'})$.
Then, we have the identities
\begin{align}\label{eq:DI2}
\frac{6}{\pi^2}\sum_{(i,i')\in I\times I'}
\sum_{u=0}^{2(h+h')-1}
L\left(
\frac{Y_{ii'}(u)}{1+Y_{ii'}(u)}
\right)
&=
2h r r',\\
\label{eq:DI3}
\frac{6}{\pi^2}\sum_{(i,i')\in I\times I'}
\sum_{u=0}^{2(h+h')-1}
L\left(
\frac{1}{1+Y_{ii'}(u)}
\right)
&=
2h' r r'.
\end{align}
\end{conjecture}
Two identities
\eqref{eq:DI2} and \eqref{eq:DI3}  are equivalent
to each other due to \eqref{eq:L2}.

Conjecture \ref{conj:DI2}
implies Conjecture \ref{conj:DI};
namely, set $X'_{r'}=A_{\ell-1}$,
and take a  {\em constant solution} $Y_{ii'}=Y_{ii'}(u)$
of $\mathbb{Y}(X_r,A_{\ell-1})$ as a function of $u$.
Then, one obtain  \eqref{eq:DI} from \eqref{eq:DI2}
using  $h'=\ell$, $r'=\ell-1$, and \eqref{eq:c1}.
Conjecture \ref{conj:DI2} was proved
for $(X_r,X'_{r'})=(A_r,A_1)$
by Frenkel-Szenes \cite{FS},
and 
for $(X_r,X'_{r'})=(\mbox{any},A_1)$
by Chapoton \cite{C}.

\begin{example}[\cite{GT1}]
(i) In the simplest case $(X_r,X'_{r'})=(A_1,A_1)$,
the identity \eqref{eq:DI2} is equivalent to
\eqref{eq:L2}.
\par
(ii) In the next simplest case $(X_r,X'_{r'})=(A_2,A_1)$,
the identity \eqref{eq:DI2} is equivalent to
the  5-term relation \eqref{eq:L3}.
\end{example}

Conjecture \ref{conj:DI2} tells us that
what is important in \eqref{eq:DI}
is, not the values of $Y^{(a)}_m$ themselves,
but rather the relations \eqref{eq:Y2}
they satisfy.
We will see below that the algebraic property
of the relations \eqref{eq:Y2} is efficiently
extracted by {\em cluster algebras} introduced 
by Fomin-Zelevinsky \cite{FZ1,FZ2,FZ4}.

\subsection{Main result}
Among several preceding results and methods
concerning Conjectures \ref{conj:period1},
and  \ref{conj:DI2},
we list  the ones
which are particularly
relevant to the present work.

\begin{itemize}
\item[(i)] Frenkel-Szenes \cite{FS} proved Conjecture \ref{conj:DI2}
for $(X_r,X'_{r'})=(A_r,A_1)$
by showing the {\em constancy property} of the
left hand side of \eqref{eq:DI2}.

\item[(ii)] Caracciolo-Gliozzi-Tateo \cite{CGT} studied the constancy property
for any pair $(X_r,X'_{r'})$. (But they did not complete the proof of
Conjecture \ref{conj:DI2}.)

\item[(iii)] Fomin-Zelevinsky \cite{FZ3}
 proved Conjecture \ref{conj:period1}
for $(X_r,X'_{r'})=(\mbox{any},A_1)$
by using the `cluster algebra like' formulation of
Y-systems and the root systems.

\item[(iv)] Chapoton \cite{C} proved Conjecture
\ref{conj:DI2} for $(X_r,X'_{r'})=(\mbox{any},A_1)$
by combining the constancy property of (i) and the result
 of (iii).

\item[(v)] Fomin-Zelevinsky \cite{FZ4} more manifestly
integrated the Y-system
$\mathbb{Y}(X_r,A_1)$ in the framework of
{\em cluster algebras with coefficients},
where the Y-system  is identified with  a system of
relations among  {\em coefficients\/} of the cluster algebras
of type $X_r$.

\item[(vi)] Keller \cite{Kel1,Kel2} proved Conjecture \ref{conj:period1}
in full generality for any pair  $(X_r,X'_{r'})$
of simply laced type  by 
using cluster algebras with coefficients
of (v)  together with their categorifications
 by the {\em cluster categories}.
\end{itemize}

We also remark that
the connection between the dilogarithm
and cluster algebras was  studied earlier by Fock-Goncharov \cite{FG}.

By combining these results, methods, and ideas,
we prove Conjecture \ref{conj:DI2}.

\begin{theorem}
\label{thm:main2}
Conjecture \ref{conj:DI2} is true for any pair
$(X_r,X'_{r'})$ of simply laced type.
\end{theorem}

\begin{corollary}
\label{cor:main}
Conjecture \ref{conj:DI} is true for any $X_r$ of simply laced type
and any $\ell\geq 2$.
\end{corollary}

\begin{remark}
The dilogarithm identities of  {\em nonsimply laced\/} type
by Kirillov \cite[Eq.~(7)]{Ki1}, properly corrected by
Kuniba \cite[Eqs.~(A.1a), (A.1c)]{Ku},
 are equally important.
We stress that 
they are {\em different\/} from another version of the identities
of nonsimply laced type obtained by the folding of simply laced one
\cite{C}.
Though the situation is more complicated than
the simply laced case,
a similar approach to the one here
is  applicable to prove the identities \cite{IIKKN1,IIKKN2}.
\end{remark}

The organization of the paper is as follows.
In section 2 we reformulate Conjecture \ref{conj:DI2}
in terms of cluster algebras.
In section 3 we study the tropical version
of the Y-system in the cluster algebra setting.
Proposition \ref{prop:monom} is a key observation through the paper.
In section 4 we prove Theorem \ref{thm:DI3},
which is equivalent to Theorem \ref{thm:main2},
by applying the method of \cite{FS} with
mixture of the ideas by \cite{CGT,C}.

\smallskip
{\em Acknowledgements.}
It is my great pleasure to thank
Atsuo Kuniba for sharing his insight
into the dilogarithm identities for
many years,
and also for useful comments on
the manuscript.


\section{Reformulation by cluster algebras}
As the first step we reformulate
Theorem \ref{thm:main2} in terms of cluster algebras.

\subsection{Cluster algebras}

Here we collect some basic definitions for cluster algebras
\cite{FZ1,FZ2,FZ4} to fix the convention and notation,
mainly following  \cite{FZ4}.

(i) {\em Matrix mutation.}
An integer matrix
$B=(B_{ij})_{i,j\in I}$  is {\em skew-symmetrizable\/}
if there is a diagonal matrix $D=\mathrm{diag}
(d_i)_{i\in I}$ with $d_i\in \mathbb{N}$
such that $DB$ is skew-symmetric.
For a skew-symmetrizable matrix $B$ and
$k\in I$, another matrix $B'=\mu_k(B)$,
called the {\em mutation of $B$ at $k$\/}, is defined by
\begin{align}
\label{eq:Bmut}
B'_{ij}=
\begin{cases}
-B_{ij}& \mbox{$i=k$ or $j=k$},\\
B_{ij}+\frac{1}{2}
(|B_{ik}|B_{kj} + B_{ik}|B_{kj}|)
&\mbox{otherwise}.
\end{cases}
\end{align}
The matrix $\mu_k(B)$ is also skew-symmetrizable.

(ii)  {\em Exchange relation of coefficient tuple.}
A {\em semifield\/} $(\mathbb{P},\oplus)$ is an
abelian multiplicative group $\mathbb{P}$ endowed with a binary
operation of addition $\oplus$ which is commutative,
associative, and distributive with respect to the
multiplication in $\mathbb{P}$ \cite{FZ4,HW}. 
For an $I$-tuple $y=(y_i)_{i\in I}$, $y_i\in \mathbb{P}$
and $k\in I$, another $I$-tuple $y'$ is defined 
by the {\em exchange relation}
\begin{align}
\label{eq:coef}
y'_i =
\begin{cases}
\displaystyle
{y_k}{}^{-1}&i=k,\\
\displaystyle
y_i \left(\frac{y_k}{1\oplus {y_k}}\right)^{B_{ki}}&
i\neq k,\ B_{ki}\geq 0,\\
y_i (1\oplus y_k)^{-B_{ki}}&
i\neq k,\ B_{ki}\leq 0.\\
\end{cases}
\end{align}

(iii)  {\em Exchange relation of cluster.}
Let $\mathbb{QP}$ be the quotient field of the group ring 
$\mathbb{Z}\mathbb{P}$ of $\mathbb{P}$,
 and let $\mathbb{QP}(u)$ be the rational function field of
algebraically independent variables $u=(u_i)_{i\in I}$
over $\mathbb{QP}$.

For an $I$-tuple $x=(x_i)_{i\in I}$ which
is a free generating set of $\mathbb{QP}(u)$
and $k\in I$, another $I$-tuple $x'$ is defined 
by the {\em exchange relation}
\begin{align}
\label{eq:clust}
x'_i =
\begin{cases}
{x_k}&i\neq k,\\
\displaystyle
\frac{y_k
\prod_{j: B_{jk}>0} x_j^{B_{jk}}
+
\prod_{j: B_{jk}<0} x_j^{-B_{jk}}
}{(1\oplus y_k)x_k}
&
i= k.\\
\end{cases}
\end{align}

(iv) {\em Seed mutation.} For the above triplet $(B,x,y)$,
called a {\em seed},  the mutation
$\mu_k(B,x,y)=(B',x',y')$  at $k$ is defined 
 by combining (i)--(iii).

(v) {\em Cluster algebra}. Fix a semifield $\mathbb{P}$
and a seed ({\em initial seed\/}) $(B,x,y)$, where
$x=(x_i)_{i\in I}$  are algebraically independent variables
over $\mathbb{Q}\mathbb{P}$.
Starting from $(B,x,y)$, iterate mutations and collect all the
seeds $(B',x',y')$.
We call  $y'$  and $y'_i$ a {\em coefficient tuple} and
a {\em coefficient}, respectively.
We call  $x'$  and $x'_i\in \mathbb{Q}\mathbb{P}(x)$, a {\em cluster} and
a {\em cluster variable}, respectively.
The {\em cluster algebra $\mathcal{A}(B,x,y)$ with
coefficients in $\mathbb{P}$} is a
$\mathbb{Z}\mathbb{P}$-subalgebra of the
rational function field $\mathbb{Q}\mathbb{P}(x)$
generated by all the cluster variables.

For further necessary definitions and information
for cluster algebras, see \cite{FZ4}.

\subsection{Matrix $B(X_r,X'_{r'})$}

For a Cartan matrix $C=(C_{ij})_{i,j\in I}$ of finite type,
we say the decomposition
 $I=I_+\sqcup I_-$ is {\em bipartite} if
\begin{align}
\label{eq:C1}
\mbox{if $C_{ij}<0$, then $(i,j)\in I_+\times I_-$
or $(i,j)\in I_-\times I_+$}.
\end{align}

{}From now on, we assume that
$X_r$ and $X'_{r'}$  are a pair of simply laced
Dynkin diagrams of finite type and that
$C=(C_{ij})_{i,j\in I}$ and $C'=(C_{i'j'})_{i',j'\in I}$
are the Cartan matrices of
$X_r$ and $X'_{r'}$  with fixed
bipartite decompositions
$I=I_+\sqcup I_-$ and $I'=I'_+\sqcup I'_-$,
respectively.
Set $\mathbf{I}=I\times I'$.
For $\mathbf{i}=(i,i')\in \mathbf{I}$,
let us write $\mathbf{i}:(++)$ if $(i,i')\in I_+\times I'_+$, {\em etc}.
Define the matrix $B=B(X_r,X'_{r'})=
(B_{\mathbf{i}\mathbf{j}})_{\mathbf{i},\mathbf{j}
\in \mathbf{I}}$ by
\begin{align}
\label{eq:Bsq}
B_{\mathbf{i}\mathbf{j}}=
\begin{cases}
-C_{ij}\delta_{i'j'}
 &
 \mathbf{i}:(-+), \mathbf{j}:(++)
\ \mbox{or}\
 \mathbf{i}:(+-), \mathbf{j}:(--),
\\
C_{ij}\delta_{i'j'}
 &
 \mathbf{i}:(++), \mathbf{j}:(-+)
\ \mbox{or}\
 \mathbf{i}:(--), \mathbf{j}:(+-),
\\
-\delta_{ij}C'_{i'j'}
 &
 \mathbf{i}:(++), \mathbf{j}:(+-)
\ \mbox{or}\
 \mathbf{i}:(--), \mathbf{j}:(-+),
\\
\delta_{ij}C'_{i'j'}
 &
 \mathbf{i}:(+-), \mathbf{j}:(++)
\ \mbox{or}\
 \mathbf{i}:(-+), \mathbf{j}:(--),
\\
0 & \mbox{otherwise}.
\end{cases}
\end{align}
The rule \eqref{eq:Bsq} is visualized in
the diagram:
\begin{align}
\label{eq:square1}
\begin{matrix}
&& _{-C}&&\\
&(+-)& \rightarrow & (--)&\\
_{-C'}&\uparrow &&\downarrow& _{-C'}\\
&(++)&\leftarrow & (-+)&\\
&& _{-C}&&\\
\end{matrix}
\end{align}
The matrix $B$ corresponds to the {\em square product of
alternating quivers\/} by \cite{Kel1}.

\begin{lemma}
The matrix $B=B(X_r,X'_{r'})$ in \eqref{eq:Bsq} is skew-symmetric and
satisfies the following conditions:
Let $\mathbf{I}=\mathbf{I}_+\sqcup \mathbf{I}_-$ with
 $\mathbf{I}_+:=(I_+\times I'_+)\sqcup(I_-\times I'_-)$
and 
$\mathbf{I}_-:=(I_+\times I'_-)\sqcup(I_-\times I'_+)$.
Then,
\begin{align}
\label{eq:B1}
\mbox{if $B_{\mathbf{i}\mathbf{j}}\neq 0$,
 then $(\mathbf{i},\mathbf{j})\in 
\mathbf{I}_+\times \mathbf{I}_-$
or $(\mathbf{i},\mathbf{j})\in \mathbf{I}_-\times \mathbf{I}_+$}.
\end{align}
Furthermore, for composed mutations
$\mu_{+}=\prod_{\mathbf{i}\in \mathbf{I}_+} \mu_{\mathbf{i}}$
and $\mu_{-}=\prod_{\mathbf{i}\in \mathbf{I}_-} \mu_{\mathbf{i}}$,
\begin{align}
\label{eq:B2}
\mu_+(B)=\mu_-(B)=-B.
\end{align}
\end{lemma}
\begin{proof}
They are easily seen in the quiver picture in
\cite[Section 8]{Kel1}.
\end{proof}

Note that $\mu_{\pm}(B)$ does not depend on the order of the product
due to \eqref{eq:B1}.


\subsection{Cluster algebra  and Y-system}

For the matrix $B=B(X_r,X'_{r'})$
in \eqref{eq:Bsq},
let $\mathcal{A}(B,x,y)$ 
be the {\em cluster algebra
 with coefficients
in the universal
semifield
$\mathbb{Q}_{\mathrm{sf}}(y)$},
where $(B,x,y)$ is the initial seed  \cite{FZ4}.
(Here we use the symbol $+$ instead of $\oplus$ 
in $\mathbb{Q}_{\mathrm{sf}}(y)$,
since it is the ordinary addition of subtraction-free
expressions of rational functions of $y$.)

To our purpose, it is natural to introduce not only
the `ring of cluster variables' but also
the `group of coefficients'.

\begin{definition}
The {\em coefficient group $\mathcal{G}(B,y)$
associated with $\mathcal{A}(B,x,y)$}
is the multiplicative subgroup of
the semifield $\mathbb{Q}_{\mathrm{sf}}(y)$ generated by all
the coefficients $y_{\mathbf{i}}'$ of $\mathcal{A}(B,x,y)$
together with $1+y_{\mathbf{i}}'$.
\end{definition}

We set $x(0)=x$, $y(0)=y$ and define 
clusters $x(u)=(x_{\mathbf{i}}(u))_{\mathbf{i}\in \mathbf{I}}$
 ($u\in \mathbb{Z}$)
 and coefficient tuples $y(u)=(y_\mathbf{i}(u))_{\mathbf{i}\in \mathbf{I}}$
 ($u\in \mathbb{Z}$)
by the sequence of mutations
\begin{align}
\label{eq:QseqADE2}
\begin{split}
\cdots
& 
 \overset{\mu_-}{\longleftrightarrow}
(B,x(0),y(0))
\overset{\mu_+}{\longleftrightarrow}
(-B,x(1),y(1))
 \overset{\mu_-}{\longleftrightarrow}
(B,x(2),y(2))
\overset{\mu_+}{\longleftrightarrow}
\cdots.
\end{split}
\end{align}

\begin{definition}
The {\em Y-subgroup
${\mathcal{G}}_Y(B,y)$
of ${\mathcal{G}}(B,y)$
associated with the sequence \eqref{eq:QseqADE2}}
is the multiplicative subgroup of
${\mathcal{G}}(B,y)$ 
generated by
$y_{\mathbf{i}}(u)$ and  $1+ y_{\mathbf{i}}(u)$
($\mathbf{i}\in \mathbf{I}, u\in \mathbb{Z}$).
\end{definition}

Let $\varepsilon:\mathbf{I} \rightarrow \{+,-\}$ be the sign
function defined by $\varepsilon(\mathbf{i})=\varepsilon$
 for $\mathbf{i}\in \mathbf{I}_{\varepsilon}$.
For $(\mathbf{i},u)\in \mathbf{I}\times \mathbb{Z}$,
we set the `parity conditions' $\mathbf{P}_{+}$ and
$\mathbf{P}_{-}$ by
\begin{align}
\label{eq:Pcond1}
\mathbf{P}_{\pm}:\quad
\varepsilon(\mathbf{i})(-1)^{u} = \pm,
\end{align}
where we identify $+$ and $-$ with $1$ and $-1$, respectively.
We write $(\mathbf{i},u):\mathbf{P}_{\varepsilon}$ if
$(\mathbf{i},u)$ satisfies the condition $\mathbf{P}_{\varepsilon}$.

\begin{lemma}[{cf. \cite[Lemma 6.18]{KNS}}]
\label{lem:y}
(1) $y_{\mathbf{i}}(u)=
y_{\mathbf{i}}(u\pm1)^{-1}$ for $(\mathbf{i},u): \mathbf{P}_{\pm}$.
\par
(2) The family $y_{\pm}=\{ y_{\mathbf{i}}(u)\mid
\mbox{$(\mathbf{i},u): \mathbf{P}_{\pm}$}
\}$ satisfies the Y-system $\mathbb{Y}(X_r,X'_{r'})$ in
${\mathcal{G}}_Y(B,y)$ by replacing $Y_{\mathbf{i}}(u)$
in $\mathbb{Y}(X_r,X'_{r'})$ with $y_{\mathbf{i}}(u)^{\pm1}$.
\end{lemma}
\begin{proof}
This follows from the exchange relation \eqref{eq:coef}.
\end{proof}

\begin{definition}
\label{def:YB}
Let $\EuScript{Y}(X_r,X'_{r'})$
be the semifield with generators
$Y_{\mathbf{i}}(u)$  ($\mathbf{i}\in \mathbf{I},u\in \mathbb{Z} $)
and the relations $\mathbb{Y}(X_r,X'_{r'})$.
Let $\EuScript{Y}^{\circ}(X_r,X'_{r'})$
be the multiplicative subgroup
of $\EuScript{Y}(X_r,X'_{r'})$
generated by
$Y_{\mathbf{i}}(u)$, $1+Y_{\mathbf{i}}(u)$
 ($\mathbf{i}\in \mathbf{I},u\in \mathbb{Z} $).
(Here we use the symbol $+$ instead of $\oplus$ 
for simplicity.)
\end{definition}

Define $\EuScript{Y}^{\circ}(X_r,X'_{r'})_{\varepsilon}$
($\varepsilon=\pm$)
to be the subgroup of $\EuScript{Y}^{\circ}(X_r,X'_{r'})$
generated by
those $Y_{\mathbf{i}}(u)$, $1+Y_{\mathbf{i}}(u)$
with  $(\mathbf{i},u):\mathbf{P}_{\varepsilon}$.
Then, we have
$\EuScript{Y}^{\circ}(X_r,X'_{r'})_+
\simeq
\EuScript{Y}^{\circ}(X_r,X'_{r'})_-
$
by $Y_{\mathbf{i}}(u)\mapsto Y_{\mathbf{i}}(u+1)$ and
\begin{align}
\EuScript{Y}^{\circ}(X_r,X'_{r'})
\simeq
\EuScript{Y}^{\circ}(X_r,X'_{r'})_+
\times
\EuScript{Y}^{\circ}(X_r,X'_{r'})_-.
\end{align}

\begin{proposition}[{cf. \cite[Theorem 6.19]{KNS}}]
\label{prop:YAA}
The group $\EuScript{Y}^{\circ}(X_r,X'_{r'})_{\pm}$
is isomorphic to
$
{\mathcal{G}}_Y(B,y)
$
by the correspondence $Y_{\mathbf{i}}(u)\mapsto y_{\mathbf{i}}(u)^{\pm1}$,
$1+Y_{\mathbf{i}}(u)\mapsto 1+ y_{\mathbf{i}}(u)^{\pm1}$
for $(\mathbf{i},u):\mathbf{P}_{\pm}$.
\end{proposition}

In summary,  the Laurent monomials in $Y_{\mathbf{i}}(u)$
and $1+Y_{\mathbf{i}}(u)$  with $(\mathbf{i},u):\mathbf{P}_{+}$
are embedded
in the coefficient group
$
{\mathcal{G}}(B,y)
$.

\subsection{Reformulation of Theorem \ref{thm:main2}}

We recall the periodicity theorem, originally conjectured by
\cite{RTV}:

\begin{theorem}[{\cite{Kel1,Kel2}}]
\label{thm:period2}
In ${\mathcal{G}}(B,y)$,
the following relations hold:

(i) Periodicity: $y_{\mathbf{i}}(u+2(h+h'))=y_{\mathbf{i}}(u)$.

(ii) Half periodicity: $y_{ii'}(u+(h+h'))=y_{\omega(i)\omega'(i')}(u)$,
where $\omega$ (resp.\ $\omega'$) is the Dynkin automorphism
of $X_r$ (resp. $X'_{r'}$) for types $A_r$, $D_{r}$ ($r:odd$), or $E_6$,
and the identity otherwise.
\end{theorem}
\begin{proof}
(i). This is due to \cite[Theorem 8.2]{Kel1} or \cite[Theorem 2.3]{Kel2}.
(ii). This is obtained by combining
 \cite[Theorem 7.13]{Kel1} and the proof
of \cite[Theorem 4.27]{IIKNS}.
\end{proof}

Let $\mathbb{R}_+$ be the semifield of the positive real numbers
by the usual multiplication and addition.
By Proposition \ref{prop:YAA}, Theorem \ref{thm:main2}
is equivalent to the following one:
\begin{theorem}
\label{thm:DI3}
Let $y_{\mathbf{i}}(u)\in \mathbb{Q}_{\mathrm{sf}}(y)$
$(\mathbf{i}\in \mathbf{I}$, $u\in \mathbb{Z})$
be as above.
Let  $\varphi: \mathbb{Q}_{\mathrm{sf}}(y)
\rightarrow \mathbb{R}_+$ be any
semifield homomorphism.
Then, the following identities hold:
\begin{align}\label{eq:DI4}
\frac{6}{\pi^2}
\sum_{
(\mathbf{i},u)\in S_+}
L\left(
\frac{\varphi(y_{\mathbf{i}}(u))}{1+\varphi(y_{\mathbf{i}}(u))}
\right)
&=
h r r',\\
\label{eq:DI5}
\frac{6}{\pi^2}
\sum_{
(\mathbf{i},u)\in S_-}
L\left(
\frac{\varphi(y_{\mathbf{i}}(u))}{1+\varphi(y_{\mathbf{i}}(u))}
\right)
&=
h' r r',
\end{align}
where $S_{\pm}=\{(\mathbf{i},u)\mid \mathbf{i}\in I,
0\leq u \leq 2(h+h')-1, \mbox{$(\mathbf{i},u):
\mathbf{P}_{\pm}$}\}$. Also,
\begin{align}\label{eq:DI6}
\frac{6}{\pi^2}
\sum_{
(\mathbf{i},u)\in H_+}
L\left(
\frac{\varphi(y_{\mathbf{i}}(u))}{1+\varphi(y_{\mathbf{i}}(u))}
\right)
&=
\frac{h r r'}{2},\\
\label{eq:DI7}
\frac{6}{\pi^2}
\sum_{
(\mathbf{i},u)\in H_-}
L\left(
\frac{\varphi(y_{\mathbf{i}}(u))}{1+\varphi(y_{\mathbf{i}}(u))}
\right)
&=
\frac{h' r r'}{2},
\end{align}
where $H_{\pm}=\{(\mathbf{i},u)\mid \mathbf{i}\in I,
0\leq u \leq (h+h')-1, \mbox{$(\mathbf{i},u):\mathbf{P}_{\pm}$}\}$.
\end{theorem}
The identities
\eqref{eq:DI4} and \eqref{eq:DI5}  are equivalent to each other by
Lemma \ref{lem:y}.
The identities
\eqref{eq:DI6} and \eqref{eq:DI7}
follow from \eqref{eq:DI4} and \eqref{eq:DI5}
by the half periodicity in Theorem \ref{thm:period2}.

We are going to prove Theorem 
\ref{thm:DI3}.

\section{Tropical Y-system}
Let us have an interlude to establish a property
of the {\em tropical Y-system} \cite{FZ4} associated with
the cluster algebra $\mathcal{A}(B,x,y)$
for $B=B(X_r,X'_{r'})$.

Let $B=(B_{ij})_{i,j\in I}$ be a general skew-symmetrizable matrix.
Let $y$ be the initial coefficient tuple
of the cluster algebra $\mathcal{A}(B,x,y)$
with coefficients in the universal
semifield $\mathbb{Q}_{\mathrm{sf}}(y)$.
The {\em tropical semifield} $\mathrm{Trop}(y)$
is an abelian multiplicative group freely generated by
the elements $y_{i}$ ($i\in I$)
with the addition $\oplus$ 
\begin{align}
\prod_{i\in I}y_{i}^{a_{i}}
\oplus
\prod_{i\in I}y_{i}^{b_{i}}
=
\prod_{i\in I}y_{i}^{\min(a_i,b_i)}.
\end{align}
The image of 
$f\in \mathbb{Q}_{\mathrm{sf}}(y)$
by the natural projection $\mathbb{Q}_{\mathrm{sf}}(y)
\rightarrow \mathrm{Trop}(y)$ is denoted by
$[f]_{\mathrm{T}}$ and called the {\em tropical evaluation
of $f$} \cite{FZ4}.

We say a (Laurent) monomial in $y=(y_i)_{\i\in I}$ is {\em positive}
if its exponents are all nonnegative and at least
one of them is positive.
A {\em negative\/} monomial is defined similarly.

\begin{lemma}
\label{lem:mutation}
Suppose that $y''$ is the coefficient tuple obtained from
the mutation of another coefficient tuple $y'$ at $k$.
Then, for any $i\neq k$,
$[y''_i]_{\mathrm{T}} = [y'_i]_{\mathrm{T}}$
 if one of the following conditions holds.
\par
(i) $B_{ki}=0$.
\par
(ii) $B_{ki}> 0$, and $[y'_k]_{\mathrm{T}}$ is negative.
\par
(iii) $B_{ki}< 0$, and $[y'_k]_{\mathrm{T}}$ is positive.
\end{lemma}
\begin{proof}
This is an immediate consequence of the exchange relation
\eqref{eq:coef}.
\end{proof}

Now we claim a key proposition in our proof of Theorem
\ref{thm:DI3}.

\begin{proposition}
\label{prop:monom}
For the cluster algebra
$\mathcal{A}(B,x,y)$ for $B=B(X_r,X'_{r'})$,
the following properties hold.
\par
(i) The tropical evaluation $[ y_{\mathbf{i}}(u)]_{\mathrm{T}}$
of $y_{\mathbf{i}}(u)$
$(\mathbf{i}\in \mathbf{I}, u\in \mathbb{Z})$ is a 
positive or negative  monomial in $y=y(0)$.

(ii) For $0\leq u \leq h'-1$ and $(\mathbf{i},u):\mathbf{P}_+$,
$[ y_{\mathbf{i}}(u)]_{\mathrm{T}}$ is a positive monomial.
For $-h\leq u \leq -1$ and $(\mathbf{i},u):\mathbf{P}_+$,
$[ y_{\mathbf{i}}(u)]_{\mathrm{T}}$ is a negative monomial.

\par
(iii) Let $N_+$ (\/resp.\ $N_-$) be the number of the positive
(resp.\ negative)
monomials
$[y_{\mathbf{i}}(u)]_{\mathrm{T}}$
$((\mathbf{i},u)\in S_+)$,
where $S_+$ is the domain in Theorem \ref{thm:DI3}.
Then,
\begin{align}
\label{eq:number}
N_+=h'rr',
\quad
N_-=hrr'.
\end{align}
\end{proposition}

The properties (i) and (iii) follow from (ii) by
the half periodicity in Theorem \ref{thm:period2}.
In the case $X'_{r'}=A_1$ or $X_r=A_1$,
i.e.,
the `level 2 case' or its level-rank dual in the original context,
Proposition \ref{prop:monom} reduces to the
known one for the {\em cluster algebra of finite type}
\cite[Proposition 10.7]{FZ4}.

Before giving a proof, it is instructive to
observe some examples.

\begin{example}
\label{example:A2}
 Let $X_r=A_1$, $I_+=\{1\}$, 
$I_-=\emptyset$,  and $X'_{r'}=A_2$,
$I'_+=\{1\}$, $I'_-=\{2\}$.
We have $h=2$ and $h'=3$.
We visualize the mutation matrix $B$ by a quiver
in the correspondence
\begin{align}
\mathbf{i} \rightarrow \mathbf{j}\quad
\Longleftrightarrow
\quad
B_{\mathbf{i}\mathbf{j}}=1.
\end{align}
Set $y_1:=y_{11}$ and $y_2:=y_{12}$.
Then, $[y_{\mathbf{i}}(u)]_{\mathrm{T}}$ for $0\leq u\leq 5$
is given as follows:
\begin{align*}
\begin{picture}(255,70)(-15,-10)
\put(0,20){\vector(0,1){20}}
\put(-15,2){\framebox(30,15)[c]{$y_1$}}
\put(-15,43){\makebox(30,15)[c]{$y_2$}}
\put(-8,-10){$y(0)$}
\put(15,25){$\leftrightarrow$}
\put(15,35){$\mu_+$}
\put(40,40){\vector(0,-1){20}}
\put(25,2){\makebox(30,15)[c]{$y_1^{-1}$}}
\put(25,43){\framebox(30,15)[c]{$y_1y_2$}}
\put(32,-10){$y(1)$}
\put(40,0)
{
\put(15,25){$\leftrightarrow$}
\put(15,35){$\mu_-$}
\put(40,20){\vector(0,1){20}}
\put(25,2){\framebox(30,15)[c]{$y_2$}}
\put(25,43){\makebox(30,15)[c]{$y_1^{-1}y_2^{-1}$}}
\put(32,-10){$y(2)$}
}
\put(80,0)
{
\put(15,25){$\leftrightarrow$}
\put(15,35){$\mu_+$}
\put(40,40){\vector(0,-1){20}}
\put(25,2){\makebox(30,15)[c]{$y_2^{-1}$}}
\put(25,43){\framebox(30,15)[c]{$y_1^{-1}$}}
\put(32,-10){$y(3)$}
}
\put(120,0)
{
\put(15,25){$\leftrightarrow$}
\put(15,35){$\mu_-$}
\put(40,20){\vector(0,1){20}}
\put(25,2){\framebox(30,15)[c]{$y_2^{-1}$}}
\put(25,43){\makebox(30,15)[c]{$y_1$}}
\put(32,-10){$y(4)$}
}
\put(160,0)
{
\put(15,25){$\leftrightarrow$}
\put(15,35){$\mu_+$}
\put(40,40){\vector(0,-1){20}}
\put(25,2){\makebox(30,15)[c]{$y_2$}}
\put(25,43){\makebox(30,15)[c]{$y_1$}}
\put(32,-10){$y(5)$}
}
\end{picture}
\end{align*}
Here, the framed variables are all the elements
in the domain $H_+$.
Certainly, we have $N_+/2=3$ and $N_-/2=2$,
which agree with \eqref{eq:number}.
Moreover, we observe that the positive monomials
occur consecutively for $0\leq u\leq 2$.
This is a consequence of \cite[Proposition 10.7]{FZ4}.
In fact, they correspond to the positive roots
$\alpha_1$, $\alpha_1+\alpha_2$, $\alpha_2$
of $A_2$.
For the later use we abbreviate the above diagram
as follows:
\begin{align*}
\begin{picture}(255,80)(-15,-15)
\put(0,20){\vector(0,1){20}}
\put(-5, -3){\framebox(10,20)[c]}
\put(-2,9){$\scriptstyle 0$}
\put(-2,1){$\scriptstyle 1$}
\put(-2,55){$\scriptstyle 1$}
\put(-2,47){$\scriptstyle 0$}
\put(-8,-15){$y(0)$}
\put(15,25){$\leftrightarrow$}
\put(15,35){$\mu_+$}
\put(40,40){\vector(0,-1){20}}
\put(38,9){$\scriptstyle 0$}
\put(36,1){-$\scriptstyle 1$}
\put(35, 43){\framebox(10,20)[c]}
\put(38,55){$\scriptstyle 1$}
\put(38,47){$\scriptstyle 1$}
\put(32,-15){$y(1)$}
\put(40,0)
{
\put(15,25){$\leftrightarrow$}
\put(15,35){$\mu_-$}
\put(40,20){\vector(0,1){20}}
\put(35, -3){\framebox(10,20)[c]}
\put(38,9){$\scriptstyle 1$}
\put(38,1){$\scriptstyle 0$}
\put(36,55){-$\scriptstyle 1$}
\put(36,47){-$\scriptstyle 1$}
\put(32,-15){$y(2)$}
}
\put(80,0)
{
\put(15,25){$\leftrightarrow$}
\put(15,35){$\mu_+$}
\put(40,40){\vector(0,-1){20}}
\put(36,9){-$\scriptstyle 1$}
\put(38,1){$\scriptstyle 0$}
\put(35, 43){\framebox(10,20)[c]}
\put(38,55){$\scriptstyle  0$}
\put(36,47){-$\scriptstyle 1$}
\put(32,-15){$y(3)$}
}
\put(120,0)
{
\put(15,25){$\leftrightarrow$}
\put(15,35){$\mu_-$}
\put(40,20){\vector(0,1){20}}
\put(35, -3){\framebox(10,20)[c]}
\put(36,9){-$\scriptstyle 1$}
\put(38,1){$\scriptstyle 0$}
\put(38,55){$\scriptstyle 0$}
\put(38,47){$\scriptstyle 1$}
\put(32,-15){$y(4)$}
}
\put(160,0)
{
\put(15,25){$\leftrightarrow$}
\put(15,35){$\mu_+$}
\put(40,40){\vector(0,-1){20}}
\put(35, 43){\makebox(10,20)[c]}
\put(38,9){$\scriptstyle 1$}
\put(38,1){$\scriptstyle 0$}
\put(38,55){$\scriptstyle 0$}
\put(38,47){$\scriptstyle 1$}
\put(32,-15){$y(5)$}
}
\end{picture}
\end{align*}
\end{example}

\begin{example}
\label{example:A3}
 Let $X_r=A_3$, $I_+=\{1,3\}$, 
$I_-=\{2\}$,  and $X'_{r'}=A_1$,
$I'_+=\{1\}$, $I'_-=\emptyset$.
We have $h=4$ and $h'=2$.
We consider mutations 
for $-6\leq u\leq 0$ by moving in the reverse direction of $u$.
The result is abbreviated in the diagram:

\begin{align*}
\begin{picture}(255,102)(-15,18)
\put(0,120)
{
\put(-3, -5){\makebox(26,10)[c]}
\put(0,-2){$\scriptstyle 0$}
\put(8,-2){$\scriptstyle 0$}
\put(16,-2){$\scriptstyle  1$}
\put(35,0){\vector(-1,0){10}}
\put(40,-2){$\scriptstyle 0$}
\put(48,-2){$\scriptstyle 1$}
\put(56,-2){$\scriptstyle 0$}
\put(65,0){\vector(1,0){10}}
\put(77, -5){\makebox(26,10)[c]}
\put(80,-2){$\scriptstyle 1$}
\put(88,-2){$\scriptstyle 0$}
\put(96,-2){$\scriptstyle 0$}
\put(40,-15){$y(-6)$}
}
\put(150,120)
{
\put(37, -5){\framebox(26,10)[c]}
\put(0,-2){$\scriptstyle 0$}
\put(8,-2){$\scriptstyle 0$}
\put(14,-2){-$\scriptstyle  1$}
\put(25,0){\vector(1,0){10}}
\put(40,-2){$\scriptstyle 0$}
\put(48,-2){$\scriptstyle 1$}
\put(56,-2){$\scriptstyle 0$}
\put(75,0){\vector(-1,0){10}}
\put(78,-2){-$\scriptstyle 1$}
\put(88,-2){$\scriptstyle 0$}
\put(96,-2){$\scriptstyle 0$}
\put(40,-15){$y(-5)$}
}
\put(120,115){$\leftrightarrow$}
\put(120,125){$\mu_+$}
\put(0,90)
{
\put(-3, -5){\framebox(26,10)[c]}
\put(0,-2){$\scriptstyle 0$}
\put(8,-2){$\scriptstyle 0$}
\put(14,-2){-$\scriptstyle  1$}
\put(35,0){\vector(-1,0){10}}
\put(40,-2){$\scriptstyle 0$}
\put(46,-2){-$\scriptstyle 1$}
\put(56,-2){$\scriptstyle 0$}
\put(65,0){\vector(1,0){10}}
\put(77, -5){\framebox(26,10)[c]}
\put(78,-2){-$\scriptstyle 1$}
\put(88,-2){$\scriptstyle 0$}
\put(96,-2){$\scriptstyle 0$}
\put(40,-15){$y(-4)$}
}
\put(-30,85){$\leftrightarrow$}
\put(-30,95){$\mu_-$}
\put(150,90)
{
\put(37, -5){\framebox(26,10)[c]}
\put(0,-2){$\scriptstyle 0$}
\put(8,-2){$\scriptstyle 0$}
\put(16,-2){$\scriptstyle  1$}
\put(25,0){\vector(1,0){10}}
\put(38,-2){-$\scriptstyle 1$}
\put(46,-2){-$\scriptstyle 1$}
\put(54,-2){-$\scriptstyle 1$}
\put(75,0){\vector(-1,0){10}}
\put(80,-2){$\scriptstyle 1$}
\put(88,-2){$\scriptstyle 0$}
\put(96,-2){$\scriptstyle 0$}
\put(40,-15){$y(-3)$}
}
\put(120,85){$\leftrightarrow$}
\put(120,95){$\mu_+$}
\put(0,60)
{
\put(-3, -5){\framebox(26,10)[c]}
\put(-2,-2){-$\scriptstyle 1$}
\put(6,-2){-$\scriptstyle 1$}
\put(16,-2){$\scriptstyle  0$}
\put(35,0){\vector(-1,0){10}}
\put(40,-2){$\scriptstyle 1$}
\put(48,-2){$\scriptstyle 1$}
\put(56,-2){$\scriptstyle 1$}
\put(65,0){\vector(1,0){10}}
\put(77, -5){\framebox(26,10)[c]}
\put(80,-2){$\scriptstyle 0$}
\put(86,-2){-$\scriptstyle 1$}
\put(94,-2){-$\scriptstyle 1$}
\put(40,-15){$y(-2)$}
}
\put(-30,55){$\leftrightarrow$}
\put(-30,65){$\mu_-$}
\put(150,60)
{
\put(37, -5){\framebox(26,10)[c]}
\put(0,-2){$\scriptstyle 1$}
\put(8,-2){$\scriptstyle 1$}
\put(16,-2){$\scriptstyle  0$}
\put(25,0){\vector(1,0){10}}
\put(40,-2){$\scriptstyle 0$}
\put(46,-2){-$\scriptstyle 1$}
\put(56,-2){$\scriptstyle 0$}
\put(75,0){\vector(-1,0){10}}
\put(80,-2){$\scriptstyle 0$}
\put(88,-2){$\scriptstyle 1$}
\put(96,-2){$\scriptstyle 1$}
\put(40,-15){$y(-1)$}
}
\put(120,55){$\leftrightarrow$}
\put(120,65){$\mu_+$}
\put(0,30)
{
\put(-3, -5){\framebox(26,10)[c]}
\put(0,-2){$\scriptstyle 1$}
\put(8,-2){$\scriptstyle 0$}
\put(16,-2){$\scriptstyle  0$}
\put(35,0){\vector(-1,0){10}}
\put(40,-2){$\scriptstyle 0$}
\put(48,-2){$\scriptstyle 1$}
\put(56,-2){$\scriptstyle 0$}
\put(65,0){\vector(1,0){10}}
\put(77, -5){\framebox(26,10)[c]}
\put(80,-2){$\scriptstyle 0$}
\put(88,-2){$\scriptstyle 0$}
\put(96,-2){$\scriptstyle 1$}
\put(40,-15){$y(0)$}
}
\put(-30,25){$\leftrightarrow$}
\put(-30,35){$\mu_-$}
\end{picture}
\end{align*}
Here, $110$, for example, represents the monomial $y_{11}y_{21}$.
The framed variables correspond to all the elements
in the domain $H_+$ modulo
half period $h+h'=6$.
Certainly, we have $N_+/2=3$ and $N_-/2=6$,
which agree with \eqref{eq:number}.
We observe that the negative monomials
occur consecutively for $-4\leq u\leq -1$.
Again, this is a consequence of \cite[Proposition 10.7]{FZ4},
and they correspond to the positive roots
of $A_3$.
\end{example}

Now we are ready to proceed to `higher level'.
The next example looks a toy example, but
completely clarifies why Proposition \ref{prop:monom} holds.

\begin{example}
\label{example:A2A3}
 Let $X_r=A_3$, $I_+=\{1,3\}$, 
$I_-=\{2\}$,  and $X'_{r'}=A_2$,
$I'_+=\{1\}$, $I'_-=\{2\}$.
We have $h=4$ and $h'=3$.
We consider mutations
for $-4\leq u\leq 3$
by moving in the both directions of $u$.
The result is shown in Figure \ref{fig:higher}.
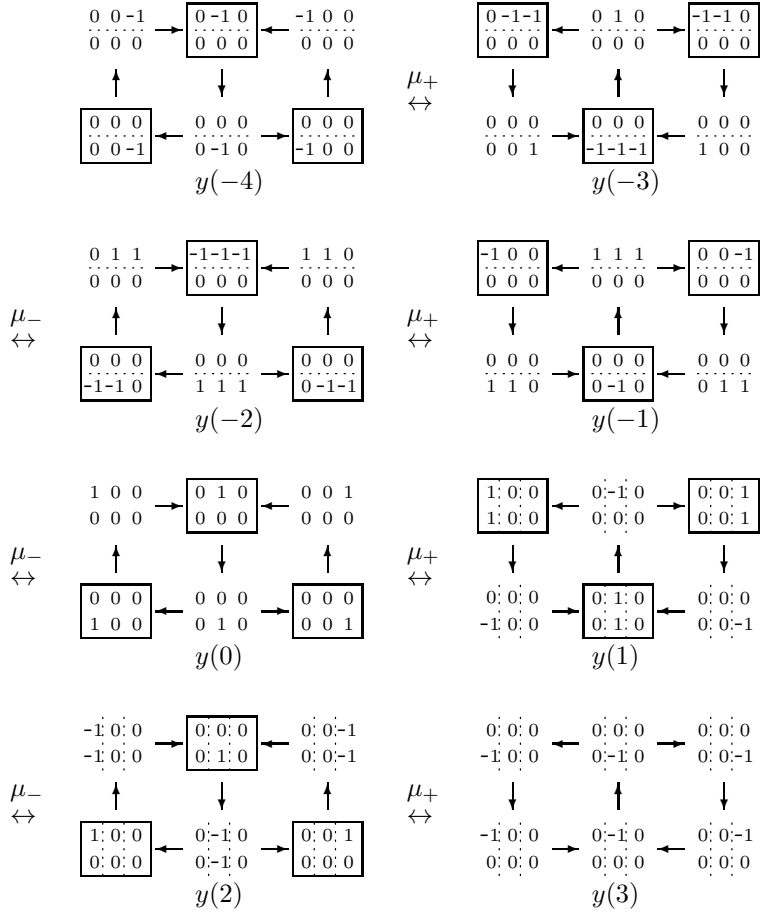
\begin{figure}[t]
\begin{picture}(255,345)(-15,-15)
\put(0,270)
{
%
\put(0,48){$\scriptstyle 0$}
\put(8,48){$\scriptstyle 0$}
\put(14,48){-$\scriptstyle  1$}
\put(0,38){$\scriptstyle 0$}
\put(8,38){$\scriptstyle 0$}
\put(16,38){$\scriptstyle  0$}
\dottedline{3}(0,45)(20,45)%
\put(10,20){\vector(0,1){10}}
\put(25,45){\vector(1,0){10}}
\put(37, 35){\framebox(26,20)[c]}
\put(40,48){$\scriptstyle 0$}
\put(46,48){-$\scriptstyle 1$}
\put(56,48){$\scriptstyle 0$}
\put(40,38){$\scriptstyle 0$}
\put(48,38){$\scriptstyle 0$}
\put(56,38){$\scriptstyle 0$}
\dottedline{3}(40,45)(60,45)
\put(50,30){\vector(0,-1){10}}
\put(75,45){\vector(-1,0){10}}
%
\put(78,48){-$\scriptstyle 1$}
\put(88,48){$\scriptstyle 0$}
\put(96,48){$\scriptstyle 0$}
\put(80,38){$\scriptstyle 0$}
\put(88,38){$\scriptstyle 0$}
\put(96,38){$\scriptstyle 0$}
\dottedline{3}(80,45)(100,45)
\put(90,20){\vector(0,1){10}}
\put(-3, -5){\framebox(26,20)[c]}
\put(0,8){$\scriptstyle 0$}
\put(8,8){$\scriptstyle 0$}
\put(16,8){$\scriptstyle  0$}
\put(0,-2){$\scriptstyle 0$}
\put(8,-2){$\scriptstyle 0$}
\put(14,-2){-$\scriptstyle  1$}
\dottedline{3}(0,5)(20,5)
\put(35,5){\vector(-1,0){10}}
\put(40,8){$\scriptstyle 0$}
\put(48,8){$\scriptstyle 0$}
\put(56,8){$\scriptstyle 0$}
\put(40,-2){$\scriptstyle 0$}
\put(46,-2){-$\scriptstyle 1$}
\put(56,-2){$\scriptstyle 0$}
\dottedline{3}(40,5)(60,5)
\put(65,5){\vector(1,0){10}}
\put(77, -5){\framebox(26,20)[c]}
\put(80,8){$\scriptstyle 0$}
\put(88,8){$\scriptstyle 0$}
\put(96,8){$\scriptstyle 0$}
\put(78,-2){-$\scriptstyle 1$}
\put(88,-2){$\scriptstyle 0$}
\put(96,-2){$\scriptstyle 0$}
\dottedline{3}(80,5)(100,5)%
\put(40,-15){$y(-4)$}
}
\put(150,270)
{
\put(-3, 35){\framebox(26,20)[c]}
\put(0,48){$\scriptstyle 0$}
\put(6,48){-$\scriptstyle 1$}
\put(14,48){-$\scriptstyle  1$}
\put(0,38){$\scriptstyle 0$}
\put(8,38){$\scriptstyle 0$}
\put(16,38){$\scriptstyle  0$}
\dottedline{3}(0,45)(20,45)%
\put(10,30){\vector(0,-1){10}}
\put(35,45){\vector(-1,0){10}}
\put(40,48){$\scriptstyle 0$}
\put(48,48){$\scriptstyle 1$}
\put(56,48){$\scriptstyle 0$}
\put(40,38){$\scriptstyle 0$}
\put(48,38){$\scriptstyle 0$}
\put(56,38){$\scriptstyle 0$}
\dottedline{3}(40,45)(60,45)
\put(50,20){\vector(0,1){10}}
\put(65,45){\vector(1,0){10}}
\put(77, 35){\framebox(26,20)[c]}
\put(78,48){-$\scriptstyle 1$}
\put(86,48){-$\scriptstyle 1$}
\put(96,48){$\scriptstyle 0$}
\put(80,38){$\scriptstyle 0$}
\put(88,38){$\scriptstyle 0$}
\put(96,38){$\scriptstyle 0$}
\dottedline{3}(80,45)(100,45)
\put(90,30){\vector(0,-1){10}}
%
\put(0,8){$\scriptstyle 0$}
\put(8,8){$\scriptstyle 0$}
\put(16,8){$\scriptstyle  0$}
\put(0,-2){$\scriptstyle 0$}
\put(8,-2){$\scriptstyle 0$}
\put(16,-2){$\scriptstyle  1$}
\dottedline{3}(0,5)(20,5)
\put(25,5){\vector(1,0){10}}
\put(37, -5){\framebox(26,20)[c]}
\put(40,8){$\scriptstyle 0$}
\put(48,8){$\scriptstyle 0$}
\put(56,8){$\scriptstyle 0$}
\put(38,-2){-$\scriptstyle 1$}
\put(46,-2){-$\scriptstyle 1$}
\put(54,-2){-$\scriptstyle 1$}
\dottedline{3}(40,5)(60,5)
\put(75,5){\vector(-1,0){10}}
%
\put(80,8){$\scriptstyle 0$}
\put(88,8){$\scriptstyle 0$}
\put(96,8){$\scriptstyle 0$}
\put(80,-2){$\scriptstyle 1$}
\put(88,-2){$\scriptstyle 0$}
\put(96,-2){$\scriptstyle 0$}
\dottedline{3}(80,5)(100,5)
\put(40,-15){$y(-3)$}
\put(-30,15){$\leftrightarrow$}
\put(-30,25){$\mu_+$}
}
\put(0,180)
{
%
\put(0,48){$\scriptstyle 0$}
\put(8,48){$\scriptstyle 1$}
\put(16,48){$\scriptstyle  1$}
\put(0,38){$\scriptstyle 0$}
\put(8,38){$\scriptstyle 0$}
\put(16,38){$\scriptstyle  0$}
\dottedline{3}(0,45)(20,45)%
%
\put(10,20){\vector(0,1){10}}
\put(25,45){\vector(1,0){10}}
\put(37, 35){\framebox(26,20)[c]}
\put(38,48){-$\scriptstyle 1$}
\put(46,48){-$\scriptstyle 1$}
\put(54,48){-$\scriptstyle 1$}
\put(40,38){$\scriptstyle 0$}
\put(48,38){$\scriptstyle 0$}
\put(56,38){$\scriptstyle 0$}
\dottedline{3}(40,45)(60,45)
\put(50,30){\vector(0,-1){10}}
\put(75,45){\vector(-1,0){10}}
%
\put(80,48){$\scriptstyle 1$}
\put(88,48){$\scriptstyle 1$}
\put(96,48){$\scriptstyle 0$}
\put(80,38){$\scriptstyle 0$}
\put(88,38){$\scriptstyle 0$}
\put(96,38){$\scriptstyle 0$}
\dottedline{3}(80,45)(100,45)
\put(90,20){\vector(0,1){10}}
\put(-3, -5){\framebox(26,20)[c]}
\put(0,8){$\scriptstyle 0$}
\put(8,8){$\scriptstyle 0$}
\put(16,8){$\scriptstyle  0$}
\put(-2,-2){-$\scriptstyle 1$}
\put(6,-2){-$\scriptstyle 1$}
\put(16,-2){$\scriptstyle  0$}
\dottedline{3}(0,5)(20,5)
\put(35,5){\vector(-1,0){10}}
\put(40,8){$\scriptstyle 0$}
\put(48,8){$\scriptstyle 0$}
\put(56,8){$\scriptstyle 0$}
\put(40,-2){$\scriptstyle 1$}
\put(48,-2){$\scriptstyle 1$}
\put(56,-2){$\scriptstyle 1$}
\dottedline{3}(40,5)(60,5)
\put(65,5){\vector(1,0){10}}
\put(77, -5){\framebox(26,20)[c]}
\put(80,8){$\scriptstyle 0$}
\put(88,8){$\scriptstyle 0$}
\put(96,8){$\scriptstyle 0$}
\put(80,-2){$\scriptstyle 0$}
\put(86,-2){-$\scriptstyle 1$}
\put(94,-2){-$\scriptstyle 1$}
\dottedline{3}(80,5)(100,5)%
\put(40,-15){$y(-2)$}
\put(-30,15){$\leftrightarrow$}
\put(-30,25){$\mu_-$}
}
\put(150,180)
{
\put(-3, 35){\framebox(26,20)[c]}
\put(-2,48){-$\scriptstyle 1$}
\put(8,48){$\scriptstyle 0$}
\put(16,48){$\scriptstyle  0$}
\put(0,38){$\scriptstyle 0$}
\put(8,38){$\scriptstyle 0$}
\put(16,38){$\scriptstyle  0$}
\dottedline{3}(0,45)(20,45)%
\put(10,30){\vector(0,-1){10}}
\put(35,45){\vector(-1,0){10}}
\put(40,48){$\scriptstyle 1$}
\put(48,48){$\scriptstyle 1$}
\put(56,48){$\scriptstyle 1$}
\put(40,38){$\scriptstyle 0$}
\put(48,38){$\scriptstyle 0$}
\put(56,38){$\scriptstyle 0$}
\dottedline{3}(40,45)(60,45)
\put(50,20){\vector(0,1){10}}
\put(65,45){\vector(1,0){10}}
\put(77, 35){\framebox(26,20)[c]}
\put(80,48){$\scriptstyle 0$}
\put(88,48){$\scriptstyle 0$}
\put(94,48){-$\scriptstyle 1$}
\put(80,38){$\scriptstyle 0$}
\put(88,38){$\scriptstyle 0$}
\put(96,38){$\scriptstyle 0$}
\dottedline{3}(80,45)(100,45)
\put(90,30){\vector(0,-1){10}}
%
\put(0,8){$\scriptstyle 0$}
\put(8,8){$\scriptstyle 0$}
\put(16,8){$\scriptstyle  0$}
\put(0,-2){$\scriptstyle 1$}
\put(8,-2){$\scriptstyle 1$}
\put(16,-2){$\scriptstyle  0$}
\dottedline{3}(0,5)(20,5)
\put(25,5){\vector(1,0){10}}
\put(37, -5){\framebox(26,20)[c]}
\put(40,8){$\scriptstyle 0$}
\put(48,8){$\scriptstyle 0$}
\put(56,8){$\scriptstyle 0$}
\put(40,-2){$\scriptstyle 0$}
\put(46,-2){-$\scriptstyle 1$}
\put(56,-2){$\scriptstyle 0$}
\dottedline{3}(40,5)(60,5)
\put(75,5){\vector(-1,0){10}}
%
\put(80,8){$\scriptstyle 0$}
\put(88,8){$\scriptstyle 0$}
\put(96,8){$\scriptstyle 0$}
\put(80,-2){$\scriptstyle 0$}
\put(88,-2){$\scriptstyle 1$}
\put(96,-2){$\scriptstyle 1$}
\dottedline{3}(80,5)(100,5)
\put(40,-15){$y(-1)$}
\put(-30,15){$\leftrightarrow$}
\put(-30,25){$\mu_+$}
}
\put(0,90)
{
%
\put(0,48){$\scriptstyle 1$}
\put(8,48){$\scriptstyle 0$}
\put(16,48){$\scriptstyle  0$}
\put(0,38){$\scriptstyle 0$}
\put(8,38){$\scriptstyle 0$}
\put(16,38){$\scriptstyle  0$}
%
\put(10,20){\vector(0,1){10}}
\put(25,45){\vector(1,0){10}}
\put(37, 35){\framebox(26,20)[c]}
\put(40,48){$\scriptstyle 0$}
\put(48,48){$\scriptstyle 1$}
\put(56,48){$\scriptstyle 0$}
\put(40,38){$\scriptstyle 0$}
\put(48,38){$\scriptstyle 0$}
\put(56,38){$\scriptstyle 0$}
%
\put(50,30){\vector(0,-1){10}}
\put(75,45){\vector(-1,0){10}}
%
\put(80,48){$\scriptstyle 0$}
\put(88,48){$\scriptstyle 0$}
\put(96,48){$\scriptstyle 1$}
\put(80,38){$\scriptstyle 0$}
\put(88,38){$\scriptstyle 0$}
\put(96,38){$\scriptstyle 0$}
%
\put(90,20){\vector(0,1){10}}
\put(-3, -5){\framebox(26,20)[c]}
\put(0,8){$\scriptstyle 0$}
\put(8,8){$\scriptstyle 0$}
\put(16,8){$\scriptstyle  0$}
\put(0,-2){$\scriptstyle 1$}
\put(8,-2){$\scriptstyle 0$}
\put(16,-2){$\scriptstyle  0$}
%
\put(35,5){\vector(-1,0){10}}
\put(40,8){$\scriptstyle 0$}
\put(48,8){$\scriptstyle 0$}
\put(56,8){$\scriptstyle 0$}
\put(40,-2){$\scriptstyle 0$}
\put(48,-2){$\scriptstyle 1$}
\put(56,-2){$\scriptstyle 0$}
%
\put(65,5){\vector(1,0){10}}
\put(77, -5){\framebox(26,20)[c]}
\put(80,8){$\scriptstyle 0$}
\put(88,8){$\scriptstyle 0$}
\put(96,8){$\scriptstyle 0$}
\put(80,-2){$\scriptstyle 0$}
\put(88,-2){$\scriptstyle 0$}
\put(96,-2){$\scriptstyle 1$}
%
\put(40,-15){$y(0)$}
\put(-30,15){$\leftrightarrow$}
\put(-30,25){$\mu_-$}
}
\put(150,90)
{
\put(-3, 35){\framebox(26,20)[c]}
\put(0,48){$\scriptstyle 1$}
\put(8,48){$\scriptstyle 0$}
\put(16,48){$\scriptstyle  0$}
\put(0,38){$\scriptstyle 1$}
\put(8,38){$\scriptstyle 0$}
\put(16,38){$\scriptstyle  0$}
\dottedline{3}(5,35)(5,55)
\dottedline{3}(13,35)(13,55)
%
\put(10,30){\vector(0,-1){10}}
\put(35,45){\vector(-1,0){10}}
\put(40,48){$\scriptstyle 0$}
\put(46,48){-$\scriptstyle 1$}
\put(56,48){$\scriptstyle 0$}
\put(40,38){$\scriptstyle 0$}
\put(48,38){$\scriptstyle 0$}
\put(56,38){$\scriptstyle 0$}
\dottedline{3}(45,35)(45,55)
\dottedline{3}(53,35)(53,55)
%
\put(50,20){\vector(0,1){10}}
\put(65,45){\vector(1,0){10}}
\put(77, 35){\framebox(26,20)[c]}
\put(80,48){$\scriptstyle 0$}
\put(88,48){$\scriptstyle 0$}
\put(96,48){$\scriptstyle 1$}
\put(80,38){$\scriptstyle 0$}
\put(88,38){$\scriptstyle 0$}
\put(96,38){$\scriptstyle 1$}
\dottedline{3}(85,35)(85,55)
\dottedline{3}(93,35)(93,55)
%
\put(90,30){\vector(0,-1){10}}
%
\put(0,8){$\scriptstyle 0$}
\put(8,8){$\scriptstyle 0$}
\put(16,8){$\scriptstyle  0$}
\put(-2,-2){-$\scriptstyle 1$}
\put(8,-2){$\scriptstyle 0$}
\put(16,-2){$\scriptstyle  0$}
\dottedline{3}(5,-5)(5,15)
\dottedline{3}(13,-5)(13,15)
%
\put(25,5){\vector(1,0){10}}
\put(37, -5){\framebox(26,20)[c]}
\put(40,8){$\scriptstyle 0$}
\put(48,8){$\scriptstyle 1$}
\put(56,8){$\scriptstyle 0$}
\put(40,-2){$\scriptstyle 0$}
\put(48,-2){$\scriptstyle 1$}
\put(56,-2){$\scriptstyle 0$}
\dottedline{3}(45,-5)(45,15)
\dottedline{3}(53,-5)(53,15)
%
\put(75,5){\vector(-1,0){10}}
%
\put(80,8){$\scriptstyle 0$}
\put(88,8){$\scriptstyle 0$}
\put(96,8){$\scriptstyle 0$}
\put(80,-2){$\scriptstyle 0$}
\put(88,-2){$\scriptstyle 0$}
\put(94,-2){-$\scriptstyle 1$}
\dottedline{3}(85,-5)(85,15)
\dottedline{3}(93,-5)(93,15)
%
\put(40,-15){$y(1)$}
\put(-30,15){$\leftrightarrow$}
\put(-30,25){$\mu_+$}
}
\put(0,0)
{
%
\put(-2,48){-$\scriptstyle 1$}
\put(8,48){$\scriptstyle 0$}
\put(16,48){$\scriptstyle  0$}
\put(-2,38){-$\scriptstyle 1$}
\put(8,38){$\scriptstyle 0$}
\put(16,38){$\scriptstyle  0$}
\dottedline{3}(5,35)(5,55)
\dottedline{3}(13,35)(13,55)
%
\put(10,20){\vector(0,1){10}}
\put(25,45){\vector(1,0){10}}
\put(37, 35){\framebox(26,20)[c]}
\put(40,48){$\scriptstyle 0$}
\put(48,48){$\scriptstyle 0$}
\put(56,48){$\scriptstyle 0$}
\put(40,38){$\scriptstyle 0$}
\put(48,38){$\scriptstyle 1$}
\put(56,38){$\scriptstyle 0$}
\dottedline{3}(45,35)(45,55)
\dottedline{3}(53,35)(53,55)
\put(50,30){\vector(0,-1){10}}
\put(75,45){\vector(-1,0){10}}
%
\put(80,48){$\scriptstyle 0$}
\put(88,48){$\scriptstyle 0$}
\put(94,48){-$\scriptstyle 1$}
\put(80,38){$\scriptstyle 0$}
\put(88,38){$\scriptstyle 0$}
\put(94,38){-$\scriptstyle 1$}
\dottedline{3}(85,35)(85,55)
\dottedline{3}(93,35)(93,55)
\put(90,20){\vector(0,1){10}}
\put(-3, -5){\framebox(26,20)[c]}
\put(0,8){$\scriptstyle 1$}
\put(8,8){$\scriptstyle 0$}
\put(16,8){$\scriptstyle  0$}
\put(0,-2){$\scriptstyle 0$}
\put(8,-2){$\scriptstyle 0$}
\put(16,-2){$\scriptstyle  0$}
\dottedline{3}(5,-5)(5,15)
\dottedline{3}(13,-5)(13,15)
\put(35,5){\vector(-1,0){10}}
\put(40,8){$\scriptstyle 0$}
\put(46,8){-$\scriptstyle 1$}
\put(56,8){$\scriptstyle 0$}
\put(40,-2){$\scriptstyle 0$}
\put(46,-2){-$\scriptstyle 1$}
\put(56,-2){$\scriptstyle 0$}
\dottedline{3}(45,-5)(45,15)
\dottedline{3}(53,-5)(53,15)
\put(65,5){\vector(1,0){10}}
\put(77, -5){\framebox(26,20)[c]}
\put(80,8){$\scriptstyle 0$}
\put(88,8){$\scriptstyle 0$}
\put(96,8){$\scriptstyle 1$}
\put(80,-2){$\scriptstyle 0$}
\put(88,-2){$\scriptstyle 0$}
\put(96,-2){$\scriptstyle 0$}
\dottedline{3}(85,-5)(85,15)
\dottedline{3}(93,-5)(93,15)
\put(40,-15){$y(2)$}
\put(-30,15){$\leftrightarrow$}
\put(-30,25){$\mu_-$}
}
\put(150,0)
{
%
\put(0,48){$\scriptstyle 0$}
\put(8,48){$\scriptstyle 0$}
\put(16,48){$\scriptstyle  0$}
\put(-2,38){-$\scriptstyle 1$}
\put(8,38){$\scriptstyle 0$}
\put(16,38){$\scriptstyle  0$}
\dottedline{3}(5,35)(5,55)
\dottedline{3}(13,35)(13,55)
\put(10,30){\vector(0,-1){10}}
\put(35,45){\vector(-1,0){10}}
\put(40,48){$\scriptstyle 0$}
\put(48,48){$\scriptstyle 0$}
\put(56,48){$\scriptstyle 0$}
\put(40,38){$\scriptstyle 0$}
\put(46,38){-$\scriptstyle 1$}
\put(56,38){$\scriptstyle 0$}
\dottedline{3}(45,35)(45,55)
\dottedline{3}(53,35)(53,55)
\put(50,20){\vector(0,1){10}}
\put(65,45){\vector(1,0){10}}
%
\put(80,48){$\scriptstyle 0$}
\put(88,48){$\scriptstyle 0$}
\put(96,48){$\scriptstyle 0$}
\put(80,38){$\scriptstyle 0$}
\put(88,38){$\scriptstyle 0$}
\put(94,38){-$\scriptstyle 1$}
\dottedline{3}(85,35)(85,55)
\dottedline{3}(93,35)(93,55)
\put(90,30){\vector(0,-1){10}}
%
\put(-2,8){-$\scriptstyle 1$}
\put(8,8){$\scriptstyle 0$}
\put(16,8){$\scriptstyle  0$}
\put(0,-2){$\scriptstyle 0$}
\put(8,-2){$\scriptstyle 0$}
\put(16,-2){$\scriptstyle  0$}
\dottedline{3}(5,-5)(5,15)
\dottedline{3}(13,-5)(13,15)
\put(25,5){\vector(1,0){10}}
%
\put(40,8){$\scriptstyle 0$}
\put(46,8){-$\scriptstyle 1$}
\put(56,8){$\scriptstyle 0$}
\put(40,-2){$\scriptstyle 0$}
\put(48,-2){$\scriptstyle 0$}
\put(56,-2){$\scriptstyle 0$}
\dottedline{3}(45,-5)(45,15)
\dottedline{3}(53,-5)(53,15)
\put(75,5){\vector(-1,0){10}}
%
\put(80,8){$\scriptstyle 0$}
\put(88,8){$\scriptstyle 0$}
\put(94,8){-$\scriptstyle 1$}
\put(80,-2){$\scriptstyle 0$}
\put(88,-2){$\scriptstyle 0$}
\put(96,-2){$\scriptstyle 0$}
\dottedline{3}(85,-5)(85,15)
\dottedline{3}(93,-5)(93,15)
\put(40,-15){$y(3)$}
\put(-30,15){$\leftrightarrow$}
\put(-30,25){$\mu_+$}
}
\end{picture}
\caption{Tropical Y-system for $X_r=A_3$ and $X'_{r'}=A_2$.}
\label{fig:higher}
\end{figure}
The framed variables corresponds to all the elements
in the domain $H_+$ modulo
half period $h+h'=7$,
and
\begin{align*}
\begin{matrix}
0&0&0\\
1&1&0\\
\end{matrix}
\end{align*}
for example, represents the monomial $y_{11}y_{12}$.
Certainly, we have $N_+/2=9$ and $N_-/2=12$,
which agree with \eqref{eq:number}.
We observe that the positive monomials
occur consecutively for $0\leq u\leq 2$,
while the negative monomials
do so for $-4\leq u\leq -1$.

Let us look at Figure \ref{fig:higher} more closely,
then we find a remarkable {\em factorization property\/}
of the tropical Y-system, which does not occur in the
nontropical Y-system.
First, see the region $0\leq u \leq 2$.
Then, the contents in the left and the right columns
for $y(u)$ mutate exactly in the same pattern as in Example
\ref{example:A2}. So does the middle column
with the other choice of bipartite decomposition of $I'$.
In particular, {\em there is no interaction in the horizontal direction}.
This is because the property (iii) in Lemma \ref{lem:mutation} is
satisfied at any mutation point $\mathbf{k}$
and any $\mathbf{i}$ horizontally adjacent to $\mathbf{k}$.
On the other hand, in the region $-4 \leq u \leq -1$,
the  contents in the lower row
mutate exactly in the same pattern as in Example
\ref{example:A3}. So does the upper row
with the other choice of bipartite decomposition of $I$.
Now, {\em there is no interaction in the vertical direction}.
Again, this is 
because the property (iii) in Lemma \ref{lem:mutation} is
satisfied at any mutation point $\mathbf{k}$
and any $\mathbf{i}$ vertically adjacent  to $\mathbf{k}$.
(Note that $(\mathbf{k},u):\mathbf{P}_-$ for mutations
in the reverse direction of $u$.)
\end{example}

\begin{proof}[Proof of Proposition \ref{prop:monom}]
It is enough to prove (ii).
We just repeat the argument in Example \ref{example:A2A3}
in a general manner.
Let us recall the definition of
the integer vector $\mathbf{d}(i,u)$
in \cite[Definition 10.2]{FZ4} (in our notation).
Let $s_1,\dots,s_r$ be the simple reflections of the Weyl
group of type $X_r$, and let
\begin{align}
t_+=\prod_{i\in I_+} s_i,
\quad
t_-=\prod_{i\in I_-} s_i.
\end{align}
Define the piecewise-linear analogue $\tau_{\pm}$
of $t_{\pm}$
acting on the set
$\Phi_{\ge -1}$ of all the positive roots and the negative
simple roots of $X_r$ by
\begin{align}
\tau_{\pm}(\alpha)=
\begin{cases}
-\alpha_i & \alpha=-\alpha_i, i\in I_{\mp}\\
t_{\pm} (\alpha)& \mbox{otherwise.}
\end{cases}
\end{align}
Then,  $\mathbf{d}(i,u)$  is defined by
\begin{align}
\mathbf{d}(i,u)=
\begin{cases}
(\tau_-\tau_+)^{u/2}(-\alpha_i)
& \mbox{$i\in I_+$ and even $u\geq 0$}\\
(\tau_-\tau_+)^{(u-1)/2}\tau_-(-\alpha_i)
& \mbox{$i\in I_-$ and odd $u\geq 0$}.
\end{cases}
\end{align}
It is known by \cite[Proposition 9.3]{FZ4}
that $\mathbf{d}(i,u)$ is a positive root of $X_r$
for $1\leq u \leq h$.
We naturally identify 
$\mathbf{d}(i,u)=\sum_{k\in I} d(i,u)_k\alpha_k$
 with the integer vector $(
d(i,u)_k)_{k\in I}$.

\par
1) {\em The case $0\leq u \leq h'-1$.}
Let $\mathbf{d}(i',u)=(d(i',u)_{k'})_{k'\in I'}$
 be the integer vector as above for $X'_{r'}$
with $I'=I'_+\sqcup I'_-$.
Let $\mathbf{\tilde{d}}(i',u)$ be the same vector
but for the opposite choice of the bipartite decomposition
of $I'$.
Forget temporarily all the arrows in the horizontal direction.
Then, one has
\begin{align}
\begin{split}
[y_{ii'}(u)]_{\mathrm{T}}
&=\prod_{k'\in I'} y_{ik'}^{d(\omega'(i'),h'-u)_{k'}},
\quad 
(\mathbf{i},u): \mathbf{P}_+, i\in I_+,\\
[y_{ii'}(u)]_{\mathrm{T}}
&=\prod_{k'\in I'} y_{ik'}^{\tilde{d}(\omega'(i'),h'-u)_{k'}},
\quad 
(\mathbf{i},u): \mathbf{P}_+, i\in I_-,
\end{split}
\end{align}
by applying \cite[Proposition 10.7]{FZ4} in our convention.
This remains to be true even in the presence of the horizontal arrows
 ({\em factorization property\/}),
because of Lemma \ref{lem:mutation} and
the positivity of vectors $\mathbf{d}(i',u)$ and $\mathbf{\tilde{d}}(i',u)$.
\par
2) {\em The case $-h \leq u \leq -1$.}
We consider mutations in the reverse direction
of $u$ at points $(\mathbf{i},u): \mathbf{P}_-$.
Let $\mathbf{d}(i,u)$
be the corresponding vector for $X_r$ with $I=I_+\sqcup I_-$,
and $\mathbf{\tilde{d}}(i,u)$ be the one
for the opposite choice of the bipartite decomposition
of $I$.
Then, repeating the same argument,
and also using Lemma \ref{lem:y} (1),
we obtain
\begin{align}
\begin{split}
[y_{ii'}(u)]_{\mathrm{T}}
&=\prod_{k\in I} y_{ki'}^{-d(\omega(i),h+u+1)_{k}},
\quad 
(\mathbf{i},u): \mathbf{P}_+, i'\in I'_-,\\
[y_{ii'}(u)]_{\mathrm{T}}
&=\prod_{k\in I} y_{ki'}^{d(\omega(i),h+u)_{k}},
\quad 
(\mathbf{i},u): \mathbf{P}_-, i'\in I'_-,\\
[y_{ii'}(u)]_{\mathrm{T}}
&=\prod_{k\in I} y_{ki'}^{-\tilde{d}(\omega(i),h+u+1)_{k}},
\quad 
(\mathbf{i},u): \mathbf{P}_+, i'\in I'_+,\\
[y_{ii'}(u)]_{\mathrm{T}}
&=\prod_{k\in I} y_{ki'}^{\tilde{d}(\omega(i),h+u)_{k}},
\quad 
(\mathbf{i},u): \mathbf{P}_-, i'\in I'_+.
\end{split}
\end{align}
This completes the proof of Proposition \ref{prop:monom}.
\end{proof}

\section{Proof of Theorem \ref{thm:DI3}}

We prove Theorem \ref{thm:DI3}
by applying the method of \cite{FS} with
mixture of ideas by \cite{CGT,C}.

The proof is divided into two steps.
First, we prove that the left hand side
of \eqref{eq:DI4} is independent of
a semifield homomorphism $\varphi$
({\em constancy or rigidity property}).
Next, we evaluate its value at the {\em $0/\infty$ limit}.

\subsection{Constancy property}

Let $A$ be a multiplicative abelian group.
The group  $A\otimes_{\mathbb{Z}}A$
is  the additive abelian group generated by $g\otimes h$
($g,h\in A$)
with relations
\begin{align}
(fg)\otimes h = f\otimes h + g\otimes h,
\quad h \otimes (fg) = h\otimes f + h \otimes g.
\end{align}
As a consequence, we also have the following relations
\begin{gather}
1\otimes h = h \otimes 1 = 0,\\
f^{-1}\otimes g = -f \otimes g,
\quad
g\otimes f^{-1} = - g\otimes f.
\end{gather}
Let $S^2A$ be the subgroup of $A\otimes_{\mathbb{Z}}A$
generated by $f\otimes f $ ($f\in A$),
and $ \bigwedge^2 A$ be the quotient
of $A\otimes_{\mathbb{Z}}A$ by $S^2A$.
In $ \bigwedge^2 A$ we use $\wedge$ instead of $\otimes$.

According to a very general theorem by \cite[Proposition 1]{FS}
(see also \cite{B,Zag1}),
the constancy property of the left hand side
of \eqref{eq:DI4} follows from the following fact.
\begin{proposition}
\label{prop:wedge}
In $\bigwedge^2 \mathbb{Q}_{\mathrm{sf}}(y)$,
we have
\begin{align}
\label{eq:wedge}
\sum_{(\mathbf{i},u)\in S_+}
y_{\mathbf{i}}(u)
\wedge
(1+y_{\mathbf{i}}(u))
=0.
\end{align}
\end{proposition}

Motivated by \cite{CGT,C},
we use the {\em $F$-polynomials\/} of \cite{FZ4} to prove
Proposition \ref{prop:wedge}.

The {\em  $F$-polynomial}
 $F_{\mathbf{i}}(u)\in \mathbb{Q}_{\mathrm{sf}}(y)$
at $(\mathbf{i},u)$ 
 ($\mathbf{i}\in \mathbf{I}$, $u\in \mathbb{Z}$)
is defined by
the specialization of $[x_{\mathbf{i}}(u)]_{\mathrm{T}}$ at $x_{\mathbf{j}}
=x_{\mathbf{j}}(0)=1$ ($\mathbf{j}\in \mathbf{I}$).
(Caution: Do not make the tropical evaluation for the addition
in $\mathbb{Q}\mathbb{P}(x)$!)
It is represented as a polynomial in $y$ with integer coefficients
due to the Laurent phenomenon \cite[Proposition 3.6]{FZ4}.

For our matrix $B=B(X_r,X'_{r'})$,
it is convenient to define the incidence matrices
$M=(M_{\mathbf{i}\mathbf{j}})_{\mathbf{i},\mathbf{j}
\in \mathbf{I}}$
and
$M'=(M'_{\mathbf{i}\mathbf{j}})_{\mathbf{i},\mathbf{j}
\in \mathbf{I}}$
as
\begin{align}
M_{\mathbf{i}\mathbf{j}}
&=
\begin{cases}
1& i\sim j, i'=j',\\
0& \mbox{otherwise},
\end{cases}
\quad
M'_{\mathbf{i}\mathbf{j}}
=
\begin{cases}
1& i= j, i'\sim j',\\
0& \mbox{otherwise}.
\end{cases}
\end{align}
Note that they are {\em symmetric} matrices.

\begin{lemma}
\label{lem:F}
 (i) For $(\mathbf{i}, u):\mathbf{P}_+$,
 the following relations hold  in $\mathbb{Q}_{\mathrm{sf}}(y)$.
\begin{align}
\label{eq:F0}
F_{\mathbf{i}}(u)
&=F_{\mathbf{i}}(u-1),
\\
\label{eq:F1}
F_{\mathbf{i}}(u-1)F_{\mathbf{i}}(u+1)
&=
\left[\frac{y_{\mathbf{i}}(u)}{1+y_{\mathbf{i}}(u)}
\right]_{\mathrm{T}}
\prod_{\mathbf{j}\in \mathbf{I}}
 F_{\mathbf{j}}(u)^{M_{\mathbf{j}\mathbf{i}}}
+
\left[\frac{1}{1+y_{\mathbf{i}}(u)}
\right]_{\mathrm{T}}
\prod_{\mathbf{j}\in \mathbf{I}}
 F_{\mathbf{j}}(u)^{M'_{\mathbf{j}\mathbf{i}}},
\\
\label{eq:F2}
y_{\mathbf{i}}(u)
&=
[y_{\mathbf{i}}(u)]_{\mathrm{T}}
\frac{
\prod_{\mathbf{j}\in \mathbf{I}}
 F_{\mathbf{j}}(u)^{M_{\mathbf{j}\mathbf{i}}}
}
{
\prod_{\mathbf{j}\in \mathbf{I}}
 F_{\mathbf{j}}(u)^{M'_{\mathbf{j}\mathbf{i}}}
}.
\end{align}
\par
(ii) Periodicity:
$
 F_{\mathbf{i}}(u+2(h+h'))= F_{\mathbf{i}}(u).
$
\par
(iii)
 Each polynomial $F_{\mathbf{i}}(u)$
has constant term 1.
\end{lemma}
\begin{proof}
(i).
The first two relations  follow from \eqref{eq:clust}.
The last one is due to \cite[Proposition 3.13]{FZ4}.
(ii). This was shown by \cite{Kel1,Kel2}.
(iii).
For $u=0$, this is true by $F_{\mathbf{i}}(0)=1$.
Then, the claim is shown by induction
on $u$, by using
\eqref{eq:F0},
 \eqref{eq:F1},
and Proposition \ref{prop:monom} (i)
(cf. \cite[Proposition 5.6]{FZ4}).
\end{proof}

\begin{remark} Lemma \ref{lem:F} (iii) is also true
by \cite[Theorem 1.7]{DWZ}.
Then, Proposition \ref{prop:monom} (i) is a consequence of
Lemma \ref{lem:F} (iii) due to \cite[Proposition 5.6]{FZ4}.
\end{remark}

By \eqref{eq:F1} and \eqref{eq:F2},
we also have,
for $(\mathbf{i}, u):\mathbf{P}_+$,
\begin{align}
\label{eq:F3}
1+y_{\mathbf{i}}(u)
&=
[1+y_{\mathbf{i}}(u)]_{\mathrm{T}}
\frac{
 F_{\mathbf{i}}(u-1) F_{\mathbf{i}}(u+1)
}
{
\prod_{\mathbf{j}\in \mathbf{I}}
 F_{\mathbf{j}}(u)^{M'_{\mathbf{j}\mathbf{i}}}
}.
\end{align}
Now, we put \eqref{eq:F2} and \eqref{eq:F3} into
\eqref{eq:wedge}, and expand it.

Firstly,
\begin{align}
\sum_{(\mathbf{i},u)\in S_+}
[y_{\mathbf{i}}(u)]_{\mathrm{T}}
\wedge
[1+y_{\mathbf{i}}(u)]_{\mathrm{T}}
=0,
\end{align}
since each monomial
$[y_{\mathbf{i}}(u)]_{\mathrm{T}}$ is either positive or negative
by Proposition \ref{prop:monom}.

Secondly, the contributions from the terms involving
only $F_{\mathbf{i}}(u)$'s vanish due to the symmetry argument of
\cite[Section 3]{CGT}.
For example,
\begin{align}
\sum_{(\mathbf{i},u)\in S_+}
\prod_{\mathbf{j}\in \mathbf{I}}
F_{\mathbf{j}}(u)^{M_{\mathbf{j}\mathbf{i}}}
\wedge
F_{\mathbf{i}}(u-1)F_{\mathbf{i}}(u+1)
=
\sum_{(\mathbf{i},u)\in S_+}
F_{\mathbf{i}}(u-1)F_{\mathbf{i}}(u+1)
\wedge
\prod_{\mathbf{j}\in \mathbf{I}}
F_{\mathbf{j}}(u)^{M_{\mathbf{j}\mathbf{i}}}
\end{align}
by changing the variables twice;
therefore, it vanishes.

Thirdly, the contribution from
the remaining five terms cancel due to
the Y-system \eqref{eq:Y2}:
\begin{align}
\sum_{(\mathbf{i},u)\in S_+}
[y_{\mathbf{i}}(u)]_{\mathrm{T}}
\wedge
F_{\mathbf{i}}(u-1)
&=
\sum_{(\mathbf{i},u)\in S_-}
[y_{\mathbf{i}}(u+1)]_{\mathrm{T}}
\wedge
F_{\mathbf{i}}(u),
\\
\sum_{(\mathbf{i},u)\in S_+}
[y_{\mathbf{i}}(u)]_{\mathrm{T}}
\wedge
F_{\mathbf{i}}(u+1)
&=
\sum_{(\mathbf{i},u)\in S_-}
[y_{\mathbf{i}}(u-1)]_{\mathrm{T}}
\wedge
F_{\mathbf{i}}(u),
\\
-\sum_{(\mathbf{i},u)\in S_+}
[y_{\mathbf{i}}(u)]_{\mathrm{T}}
\wedge
\prod_{\mathbf{j}\in \mathbf{I}}
F_{\mathbf{j}}(u)^{M'_{\mathbf{j}\mathbf{i}}}
&=
\sum_{(\mathbf{i},u)\in S_-}
\prod_{\mathbf{j}\in \mathbf{I}}
[y_{\mathbf{j}}(u)]_{\mathrm{T}}^{-M_{\mathbf{j}\mathbf{i}}'}
\wedge
F_{\mathbf{i}}(u),
\\
-\sum_{(\mathbf{i},u)\in S_+}
[1+y_{\mathbf{i}}(u)]_{\mathrm{T}}
\wedge
\prod_{\mathbf{j}\in \mathbf{I}}
F_{\mathbf{j}}(u)^{M_{\mathbf{j}\mathbf{i}}}
&=
\sum_{(\mathbf{i},u)\in S_-}
\prod_{\mathbf{j}\in \mathbf{I}}
[1+y_{\mathbf{j}}(u)]_{\mathrm{T}}^{-M_{\mathbf{j}\mathbf{i}}}
\wedge
F_{\mathbf{i}}(u),
\\
-\sum_{(\mathbf{i},u)\in S_+}
[1+y_{\mathbf{i}}(u)]_{\mathrm{T}}
\wedge
\prod_{\mathbf{j}\in \mathbf{I}}
F_{\mathbf{j}}(u)^{-M'_{\mathbf{j}\mathbf{i}}}
&=
\sum_{(\mathbf{i},u)\in S_-}
\prod_{\mathbf{j}\in \mathbf{I}}
[1+y_{\mathbf{j}}(u)]_{\mathrm{T}}^{M'_{\mathbf{j}\mathbf{i}}}
\wedge
F_{\mathbf{i}}(u).
\end{align}
This completes the proof of Proposition \ref{prop:wedge}.

\begin{remark}
Eq.~\eqref{eq:wedge} was studied in \cite[Eq.~(3.32)]{CGT}
by using the parametrization of $y_{\mathbf{i}}(u)$
by the {\em  T-system}. However,
it is known that this does not give a `general' solution
of the Y-system \cite[remark after Proposition 3.8]{IIKNS}.
Therefore, we use $F$-polynomials instead of the T-system.
\end{remark}

\subsection{Evaluation in the $0/\infty$ limit}
Following \cite{C}, we evaluate
the value
 of the left hand side
of \eqref{eq:DI4}
in the limit such that
each  $\varphi(y_{\mathbf{i}}(u))$
goes to either zero or infinity
(the {\em $0/\infty$ limit}).
Then, the value is equal to
the number of the variables which 
go to infinity due to \eqref{eq:L1}.

Thanks to Proposition
\ref{prop:monom},
we already have such a limit at hand.
Take the one parameter family 
of semifield homomorphisms $\varphi_t:\mathbb{Q}_{\mathrm{sf}}(y)
\rightarrow \mathbb{R}_+$ ($0<t<1$)
defined by $\varphi_t(y_{\mathbf{i}})=t$
for any $\mathbf{i}\in I$.
Then, in the limit $t\rightarrow 0$,
$\varphi_t(y_{\mathbf{i}}(u))$ is 0
if $[y_{\mathbf{i}}(u)]_{\mathrm{T}}$ is positive
and $\infty$
if $[y_{\mathbf{i}}(u)]_{\mathrm{T}}$ is negative,
due to \eqref{eq:F2} and Lemma \ref{lem:F} (iii).
Therefore, the value of
 the left hand side
of \eqref{eq:DI4} is $N_-=hrr'$ by
Proposition \ref{prop:monom}.

This completes the proof
of Theorem \ref{thm:DI3}.

\end{document}